\documentclass[a4, 12pt]{amsart}
\usepackage{amssymb}
\usepackage{amstext}
\usepackage{amsmath}
\usepackage{amscd}
\usepackage{latexsym}
\usepackage{amsfonts}

\theoremstyle{plain}
\newtheorem{thm}{Theorem}[section]
\newtheorem*{thm*}{Theorem}
\newtheorem*{cor*}{Corollary}

\newtheorem{prop}[thm]{Proposition}

\newtheorem{cor}[thm]{Corollary}

\newtheorem*{claim*}{Claim}

\theoremstyle{definition}

\newtheorem{strategy}[thm]{Main examples}
\newtheorem{ex}[thm]{Example}
\newtheorem{rem}[thm]{Remark}
\newtheorem{fact}[thm]{Fact}

\newtheorem{prob}[thm]{Problem}

\theoremstyle{remark}

\numberwithin{equation}{thm}

\def\Proj{\mathrm{Proj}}

\def\mod{\mathrm{mod}}

\def\Ker{\mathrm{Ker}}

\def\a{\mathfrak a}

\def\c{\mathfrak c}
\def\e{\mathrm{e}}
\def\m{\mathfrak m}
\def\n{\mathfrak n}

\def\p{\mathfrak p}
\def\q{\mathfrak q}

\def\H{\mathrm{H}}

\def\cl{\overline}

\newcommand{\rma}{\mathrm{a}}
\newcommand{\rmb}{\mathrm{b}}
\newcommand{\rmc}{\mathrm{c}}

\newcommand{\rme}{\mathrm{e}}

\newcommand{\rmH}{\mathrm{H}}

\newcommand{\rmS}{\mathrm{S}}

\newcommand{\rmU}{\mathrm{U}}

\newcommand{\calR}{\mathcal{R}}

\newcommand{\fkm}{\mathfrak{m}}

\newcommand{\fkq}{\mathfrak{q}}

\def\depth{\mathrm{depth}}

\def\Ass{\mathrm{Ass}}

\def\gr{\mathrm{gr}}

\def\M{{\mathcal M}}

\begin{document}

\setlength{\baselineskip}{20pt}
\title{Variation of the first Hilbert coefficients of parameters  with  a common integral closure}
\pagestyle{plain}
\author{L. Ghezzi}
\address{Department of Mathematics, New York City College of Technology-Cuny, 300 Jay Street, Brooklyn, NY 11201, U. S. A.}
\email{lghezzi@citytech.cuny.edu}
\author{S. Goto}
\address{Department of Mathematics, School of Science and Technology, Meiji University, 1-1-1 Higashi-mita, Tama-ku, Kawasaki 214-8571, Japan}
\email{goto@math.meiji.ac.jp}
\author{J. Hong}
\address{Department of Mathematics, Southern Connecticut State University, 501 Crescent Street, New Haven, CT 06515-1533, U. S. A.}
\email{hongj2@southernct.edu}
\author{K. Ozeki}
\address{Department of Mathematics, School of Science and Technology, Meiji University, 1-1-1 Higashi-mita, Tama-ku, Kawasaki 214-8571, Japan}
\email{kozeki@math.meiji.ac.jp}
\author{T.T. Phuong}
\address{Department of Information Technology and Applied Mathematics,
Ton Duc Thang University, 98 Ngo Tat To Street, Ward 19, Binh Thanh District,
Ho Chi Minh City, Vietnam}
\email{sugarphuong@gmail.com}
\author{W.V. Vasconcelos}
\address{Department of Mathematics, Rutgers University, 110 Frelinghuysen Rd, Piscataway, NJ 08854-8019, U. S. A.}
\email{vasconce@math.rutgers.edu}

\thanks{{AMS 2010 {\em Mathematics Subject Classification:}
Primary 13H15, Secondary 13H10.}\\The first author is partially supported by a grant from the City University of New York PSC-CUNY Research Award Program-41. The second author is partially supported by Grant-in-Aid for Scientific Researches (C) in Japan (19540054).
The fourth author is supported by a grant from MIMS (Meiji Institute for Advanced Study of Mathematical Sciences).
The fifth author is supported by JSPS Ronpaku (Dissertation of PhD) Program.
The last author is partially supported by the NSF}

\begin{abstract}
A problem posed by Wolmer V. Vasconcelos \cite{V2} on the variation of the
first Hilbert coefficients of parameter ideals  with a common  integral closure in a local ring  is studied. Affirmative answers are given and counterexamples are explored as well.
\end{abstract}

\maketitle


\section{Introduction}
The purpose of this paper is to study a problem posed by the last author \cite{V2} on the variation, in a given local ring,  of the first Hilbert coefficients of parameter ideals  with a common  integral closure. To state the problem and the results as well, first of all let us fix our notation and terminology.

Let $A$ be a Noetherian local ring with  maximal ideal $\m$ and $d = \operatorname{dim} A > 0$. Let $\ell_A(M)$ denote, for an $A$-module $M$, the length of $M$. Then for each $\m$--primary ideal $I$ in $A$ and for each finitely generated $A$--module $M $ with $s = \dim_AM  \ge0$, we have integers $\{\e_I^i(M)\}_{0 \le i \le s}$ such that the equality $$\ell_A(M/I^{n+1}M)={\e}_I^0(M)\binom{n+s}{s}-\e_I^1(M)\binom{n+s-1}{s-1}+\cdots+(-1)^s{\e}_I^s(M)$$
holds true for all integers $n \gg 0$, which we call the Hilbert coefficients of $M$ with respect to $I$. The leading coefficient $\e_I^0(M)$ is called the multiplicity of $M$ with respect to $I$ and plays a very  important role in the analysis of  singularity  in $M$. In this paper we are mainly interested in the case where $M = A$ and $I = Q$ is a parameter ideal in $A$, that is an ideal generated by a system of parameters in $A$. Let $\cl{\a}$ denote, for an ideal $\a$ in $A$,  the integral closure of $\a$.

With this notation the last author  \cite{V2} proved that for every $\m$--primary ideal $I$ in $A$ {\it the set}
$$\Lambda (I) = \{\e_Q^1(A) \mid Q ~\text{is~a~parameter~ideal~in}~ A~\text{and}~ \cl{Q} = \cl{I}\}$$ 
{\it is} $finite$. Added to it, he raised the following, which is the main target of the present research.

\begin{prob}[\cite{V2}]\label{ques}
Is $\sharp \Lambda (I) = 1$ for every $\m$-primary ideal $I$ in  $A$?
\end{prob}

This problem has led the authors to the researches \cite{GhGHOPV1, GhGHOPV2, GO1, GO2, GO3}  on  the finiteness of the set $$\Lambda_1(A) = \{\e_Q^1(A) \mid Q ~\text{is~a~parameter~ideal~in}~ A~\}.$$ Among many results they proved that  $A$ is a generalized Cohen-Macaulay (resp. Buchsbaum)  ring if and only if  $\Lambda_1(A)$ is a finite set (resp. $\sharp \Lambda_1(A) = 1$), provided $A$ is unmixed, that is $\operatorname{dim} \widehat{A}/\p = d$ for every $\p \in \Ass \widehat{A}$, where $\widehat{A}$ denotes the $\m$--adic completion of $A$. We  now come back to the starting point, that is  Problem \ref{ques}, of our researches and are going to settle  it.

Let us now state our main results, explaining how this paper is organized.

Because the value $\e_Q^1(A) = -1$ is the greatest one among possible values of $\e_Q^1(A)$ for parameter ideals $Q$ in  non-Cohen-Macaulay unmixed local rings $A$ (\cite[Theorem 2.1]{GhGHOPV1}), the condition that $\e_Q^1(A) = -1$ for some parameter ideal $Q$ is a rather strong restriction. First of all, in Section 3 we shall study the structure of certain local rings $A$ which contain parameter ideals $Q$ with $\e_Q^1(A) = -1$ (Theorem \ref{1.1}). We actually have $\Lambda (\m) = \{-1\}$ for the local rings $(A, \m)$ explored in Theorem \ref{1.1}, which gives an affirmative answer to Problem \ref{ques} in the special case.

For each parameter ideal $Q$ in $A$ let $$\calR (Q) = A[Qt] \subseteq A[t] \ \ \text{and}\ \ \gr_Q(A) = \calR (Q)/Q\calR (Q)$$ (here $t$ denotes  an indeterminate over $A$), which we call the Rees algebra and the associated graded ring of $Q$ respectively. In Section 4 we  shall study some affirmative cases for Problem  \ref{ques}. In particular, we will prove the following, exploring an example satisfying  condition $(\sharp)$ addressed in it.

\begin{thm}\label{1.2} Assume that the residue class field $A/\m$ of $A$ is infinite. 
Let $I$ be an $m$-primary ideal in $A$ and assume that $(\sharp)$ for every minimal reduction $Q$ of $I$ the scheme $\operatorname{Proj} \calR (Q)$ is locally Cohen-Macaulay. Then $\e_Q^1(A)$ is constant and independent of the choice of minimal reductions $Q$ of $I$.
\end{thm}

In Section 5 we shall study the case where the answer is negative. To state the result, let us briefly recall the notion of Buchsbaum ring and that of generalized Cohen-Macaulay ring. We say that $A$ is a {\it Buchsbaum} ring, if the difference $\ell_A(A/Q) - \rme_Q^0(A)$ is constant and independent of the choice of parameter ideals $Q$ in $A$ (\cite{SV1}).  We say that $A$ is a {\it generalized Cohen-Macaulay} ring if 
$$
\sup_Q[\ell_A(A/Q) - \rme_Q^0(A)] < \infty
$$
(\cite{STC}), where $Q$ runs over parameter ideals in $A$. This condition is equivalent to saying that the local cohomology modules $\rmH_\fkm^i(A)$ of $A$ with respect to $\m$ are finitely generated  for all $i \ne d$. When this is the case, we have the equality 
$$
\sup_Q [\ell_A(A/Q) - \rme_Q^0(A)] = \sum_{j=0}^{d-1} \binom{d-1}{j} \ell_A(\rmH_\m^j(A)) := \Bbb I(A),
$$
which we call the Buchsbaum invariant of $A$.

For a moment, suppose that $A$ is a generalized Cohen-Macaulay ring and let $Q=(a_1, a_2, \ldots, a_d)$ be a parameter ideal in $A$. Then  we say that $Q$ is standard if
$$
\ell_A(A/Q) - \rme_Q^0(A) = \Bbb I (A).
$$
This condition is equivalent to saying that the system $a_1, a_2, \ldots , a_d$ of generators of $Q$  forms a strong $d$-sequence in $A$ in any order (\cite{STC}), that is for all integers $n_1, n_2, \ldots , n_d > 0$, $a_1^{n_1} , a_2^{n_2} , \ldots , a_d^{n_d}$ forms a $d$-sequence in $A$ in  any order (see \cite{H} for the notion of {\it $d$--sequence}). We say that an $\m$--primary ideal $I$ in $A$ is standard, if every parameter ideal $Q$ contained in $I$ is standard. Since every standard parameter ideal $Q$ is standard also in this sense  (\cite[Proposition 3.1, Corollary 3.3]{T}), for each generalized Cohen-Macaulay ring $A$ one can find an integer $\ell \gg 0$ such that $\fkm^\ell$ is standard. Hence $A$ is a Buchsbaum ring if and only if $A$ is a generalized Cohen-Macaulay ring in which the maximal ideal $\m$ is standard.

With this notation and terminology our negative answer is stated as follows, which we shall prove in Section 5.

\begin{thm}\label{1.3} Suppose that  $A$ is a generalized Cohen-Macaulay ring with infinite residue class field. Let $Q$ be a standard parameter ideal in $A$ and put $I = \cl{Q}$. If $I$ is not a standard ideal, then there exists at least one  minimal reduction $Q'$ of $I$ such that
$$0 > \e_{Q'}^1(A) > \e_Q^1(A).$$
\end{thm}

In Section 5 we will explore  concrete examples of the ideals $Q$ satisfying the conditions required in Theorem \ref{1.3}. For the purpose we have to compute $\e_Q^1(A)$ precisely  in the case where $Q$ is a parameter ideal in a two-dimensional generalized Cohen-Macaulay ring $A$. So, let us summarize in Section 2 a certain  primitive method of computing $\e_Q^1(A)$. The same method of computation is partially explored, in a slightly different way, also in \cite[Example 3.8]{GhHV} and \cite[Section 3]{MV}.

The constancy of the values $\e_Q^1(A)$ is related to that  of the rank (or, the multiplicity) of Sally modules $$\mathrm{S}_Q(I)= \bigoplus_{n \ge 1}I^{n+1}/Q^nI$$ with $ I = \overline{Q}$. In Section 5 we will explore this phenomenon, computing the rank of Sally modules of certain {\it bad} ideals.

The integral closures of ideals given in Section 5 as counterexamples for Problem \ref{ques} are not equal to the maximal ideal $\m$. In the final section 6 we will  construct counterexamples for Problem \ref{ques} with $I=\m$, exploring Sally modules of bad ideals.

In what follows, unless otherwise specified, let $A$ denote a Noetherian local ring with maximal ideal $\m$ and $d = \operatorname{dim} A > 0$. Let $\rmH_\m^i(*)$~$(i \in \Bbb Z)$  be the $i$--th local cohomology functor of $A$ with respect $\m$. For each finitely generated $A$--module $M$ let $\mu_A(M)$ and $\ell_A(M)$ stand for the number of elements in a minimal system of generators for $M$ and the length  of $M$ respectively.


\section{A method to compute $\e_Q^1(A)$}
In this section we assume that  $\operatorname{dim} A= 2$ and that $A$ is a homomorphic image of a Gorenstein  ring, say $A = R/\a$ with $R$ a Gorenstein local ring and $\a$ an ideal in it. We assume that $A$ is unmixed. Hence $\H_\m^1(A)$ is a finitely generated $A$-module (\cite[Proposition 2.2 (1)]{GhGHOPV1}). Let $Q = (a,b)$ be a parameter ideal in $A$. Then, thanks to a lemma of Davis \cite[Theorem 124]{K}, we get a regular sequence $x, y$ in $R$ so that $$a = x ~\mod ~\a \ \ \text{and} \ \ b = y ~\mod ~\a.$$ We put $\q=(x,y)R$; hence $Q = \q A$. Let $B = \mathrm{Hom}_A(\mathrm{K}_A, \mathrm{K}_A)$ be the endomorphism ring of the canonical module $\mathrm{K}_A$ (hence $B$ is the Cohen-Macaulayfication of $A$ in the sense of \cite{AG})  and look at the  exact sequence
$$(E) \ \ \ \ 0 \to A \overset{\varphi}{\to} B \to C\to 0$$
of $A$-modules, where $\varphi (\alpha)$ is defined, for each $\alpha \in A$, to be the homothety of $\alpha$. Then, since $\operatorname{depth}_A\mathrm{K}_A = 2$,  $B$ is a Cohen-Macaulay $A$-module with $\operatorname{dim}_AB = 2$ and we get $C \cong \mathrm{H}_\m^1(A)$ as $A$-modules (\cite[Theorem 3.2, Proof of Theorem 4.2]{A}, \cite[Theorem 1.6]{AG}). Let $n \ge 0$ be an integer and let $\Bbb M$ denote the following $n+1$ by $n+2$ matrix
$$
\begin{pmatrix}
x&y&0&0&0&\ldots&0\\
0&x&y&0&0&\ldots&0\\
0&0&x&y&0&\ldots&0\\
&&&\vdots&&\\
0&0&\ldots&0&0&x&y
\end{pmatrix}.
$$ Then the ideal $\q^{n+1}$ is generated by the maximal minors of the matrix $\Bbb M$ and, thanks to the theorem of Hilbert--Burch (\cite[Exercises 8, p. 148]{K}), the $R$-module $R/\q^{n+1}$ has a minimal free resolution of the form
$$\Bbb F :\ \ \ 0 \longrightarrow F_2 =R^{n+1} \overset{^t\Bbb M}{\longrightarrow} F_1 = R^{n+2} \overset{\partial}{\longrightarrow} F_0 =R \longrightarrow R/\q^{n+1} \longrightarrow 0, $$ in which  the homomorphism $\partial$ is  defined by $$\partial (\mathbf{e}_j) = (-1)^j{\cdot}\mathrm{det}~{\Bbb M}_j $$ for all $1 \le j \le n+2$, where ${\Bbb M}_j$ denotes the matrix obtained by deleting from ${\Bbb M}$ the $j$-th column and $\{\mathbf{e}_j \}_{1 \le j \le n+2}$ denotes the standard basis of $R^{n+2}$. Consequently, for each $R$-module $X$ $\operatorname{Tor}_j^R(R/\q^{n+1}, X)$ is computed as the $j$-th homology module of the complex
$${\Bbb F}\otimes_R X :\ \ \   0 \to X^{n+1}=F_2\otimes_RX \overset{^t{\Bbb M} \otimes_R1_X}{\longrightarrow} X^{n+2}=F_1\otimes_RX \overset{\partial\otimes_R1_X}{\longrightarrow} X = F_0\otimes_RX \longrightarrow 0.$$ Setting $X = B$, we therefore  have $\mathrm{Tor}_i^R(R/\q^{n+1}, B) = (0)$ for all $i \ge 1$ (\cite[Theorem 9.1.6]{BH}), since the ideal $\q =(x,y)R$ is generated by a $B$-regular sequence of length $2$, so that exact sequence $(E
)$ gives rise to the following exact sequence
$$0 \to \mathrm{Tor}_1^R(R/\q^{n+1}, C) \to A/Q^{n+1} \to B/Q^{n+1}B \to C/Q^{n+1}C \to 0,$$ whence $$(1) \ \ \ \ \ell_A(A/Q^{n+1}) = \ell_A(B/Q^{n+1}B) + \ell_A(\mathrm{Tor}_1^R(R/\q^{n+1}, C)) - \ell_A(C/Q^{n+1}C)$$ for all $n \ge 0$. On the other hand, since the alternating sum of the length of homology modules of the complex $$\Bbb F \otimes_R C : \ \ \ 0 \to C^{n+1}=F_2\otimes_RC \overset{^t{\Bbb M} \otimes_R1_C}{\longrightarrow} C^{n+2}=F_1\otimes_RC \overset{\partial\otimes_R1_C}{\longrightarrow} C = F_0\otimes_RC\longrightarrow 0$$ is $0$,  we get
$$ \ell_R(\mathrm{Tor}_1^R(R/\q^{n+1}, C)) = \ell_R(\mathrm{Tor}_2^R(R/\q^{n+1}, C)) + \ell_A(C/Q^{n+1}C).$$
 Hence by equation (1) we have for all $n \ge 0$ that
$$(2) \ \ \ \ell_A(A/Q^{n+1}) =\e_Q^0(A)\binom{n+2}{2} + \ell_R(\mathrm{Tor}_2^R(R/\q^{n+1},C)),$$
 because  $\e_Q^0(A) = \e_Q^0(B) = \ell_A(B/QB)$ (see exact sequence $(E)$; recall that $B$ is a Cohen-Macaulay $A$-module with $\operatorname{dim}_AB  = 2$ and $\ell_A(C) < \infty$) and $\ell_A(B/Q^{n+1}B) = \ell_A(B/QB)\binom{n+2}{2}$ for all $n \ge 0$. Notice that 
$$\mathrm{Tor}_2^R(R/\q^{n+1},C) \cong \Ker (C^{n+1} \overset{^t{\Bbb M}}{\longrightarrow} C^{n+2}),$$ that is
$$\mathrm{Tor}_2^R(R/\q^{n+1},C) \cong \left\{\left.\begin{pmatrix}
\alpha_0\\
\alpha_1\\
\vdots\\
\alpha_{n}
\end{pmatrix} \in C^{n+1} \right| a\alpha_{i} + b\alpha_{i-1}= 0\ \operatorname{for~all} \ 0 \le i \le n+1
\right\},$$ where $\alpha_{-1} = \alpha_{n+1} = 0$ for convention.

Summarizing these observations, we get the following.

\begin{prop}\label{GhHV}
Let $$T_n = \left\{\left.\begin{pmatrix}
\alpha_0\\
\alpha_1\\
\vdots\\
\alpha_{n}
\end{pmatrix} \in C^{n+1} \right| a\alpha_{i} + b\alpha_{i-1}= 0\ \operatorname{for~all} \ 0 \le i \le n+1
\right\}$$ for each $n \ge 0$. Then the following assertions hold true.
\begin{enumerate}
\item[$(1)$] $\ell_A(A/Q^{n+1}) = \e_Q^0(A)\binom{n+2}{2}
 + \ell_A(T_n)$ for all $n \ge 0$.
\item[$(2)$] $-\ell_A(C) \le \e_Q^1(A) \le -\ell_A((0):_CQ)$.
\item[$(3)$] Suppose $aC=(0)$. Then $\e_Q^1(A) = -\ell_A((0):_Cb) = -\ell_A(C/bC)$ and $\e_Q^2(A) = 0$.
\item[$(4)$] $($\cite[Example 3.8]{GhHV}, \cite[Section 3]{MV}$)$ Suppose $QC=(0)$. Then $\e_Q^1(A) = -\ell_A(C)$ and $\e_Q^2(A) = 0$.
\end{enumerate}
\end{prop}

\begin{proof} See equation (2) for assertion (1). We have $$\ell_A((0):_CQ)(n+1) \le \ell_A(T_n) \le \ell_A(C)(n+1),$$ since $[(0) :_C Q]^{n+1} \subseteq T_n \subseteq C^{n+1}$, which proves  assertion (2).  If $aC = (0)$, then $T_n = [(0):_Cb]^{n+1}$, so that  $$\ell_A(A/Q^{n+1}) = \e_Q^0(A)\binom{n+2}{2} + \ell_A((0):_Cb)\binom{n+1}{1}$$ by assertion (1), from which  assertion (3) follows, because $\ell_A((0):_Cb) = \ell_A(C/bC)$. Assertion (4) is now clear.
\end{proof}

We explore two examples in order to show how we use Proposition \ref{GhHV}. Let $$\Lambda_i(A) = \{\e_Q^i(A) \mid Q ~\text{is ~a ~parameter ~ideal ~in}~ A \}$$ for $0 \le i \le d = \dim A$. For each $a \in \m$ we denote by $\rmU (a)$ the unmixed component of the ideal $(a)$. When $A$ is a generalized Cohen-Macaulay ring with $d = \dim A \ge 2$ and $a, b$ is a part of a system of parameters in $A$, one has $\rmU (a)/(a) = \rmH_\m^0(A/(a))$  (\cite{STC}), so that  $$\rmU (a)  = \bigcup_{n \ge 1}[(a) : b^n].$$

We begin with the following.

\begin{prop}\label{4.7}
Assume that $A$ is a generalized Cohen-Macaulay ring with $\operatorname{dim} A = 2$ and $\operatorname{depth} A = 1$. Let $a,b$ be a system of parameters in $A$. We then have the following.
\begin{enumerate}
\item[$(1)$] Suppose $b\rmH_\m^1(A) = (0)$. Then $a,b$ forms a $d$-sequence in $A$.
\item[$(2)$] $\rmU (a) \subseteq \overline{(a)}$.
\end{enumerate}
\end{prop}

\begin{proof} 
(1) Let $\c = (0) :_A{\rmH_\m^1(A)}$. The element $a$ is $A$--regular, since $A$ is a generalized Cohen-Macaulay ring with $\depth A > 0$. Therefore, the exact sequence $$0 \to A \overset{a}{\to} A \to A/(a) \to 0$$ gives rise to the long  exact sequence
$$0 \to \rmH_\m^0(A/(a)) \to \rmH_\m^1(A) \overset{a}{\to} \rmH_\m^1(A) \to \rmH_\m^1(A/(a)) \to \ldots$$ of local cohomology modules, from which we get $$[(a) : b^2]/(a) \subseteq \rmU (a)/(a) = \H_\m^0(A/(a)) \cong (0) :_{\H_\m^1(A)} a \subseteq \rmH_\m^1(A).$$ Hence $\frak{c}\left[((a):b^2)/(a)\right] =(0)$, so that $(a) : b^2 \subseteq (a) : \frak{c} \subseteq (a) : b$. Thus $a,b$ forms a $d$-sequence in $A$.

(2) Let $B$ denote the Cohen-Macaulayfication of $A$ in the sense of \cite{AG, G1}. We then have $\rmU (a) B = aB$, since $a,b$ forms a $B$--regular sequence. Therefore  $\rmU (a)\subseteq \overline{(a)}$, because $B$ is a module-finite extension of $A$.
\end{proof}

\begin{ex} Let $R=k[[X,Y,Z,W]]$ be the formal power series ring over a field $k$.  We look at the local ring $$A = R/[(X,Y)^\ell \cap (Z,W)],$$ where $\ell \ge 1$ is an integer.
 Then $A$ is a $2$-dimensional generalized Cohen-Macaulay ring with $\operatorname{depth} A = 1$ and we have the following.
\begin{enumerate}
\item[$(1)$] Let $a,b$ be a system of parameters in $A$. Then $a,b$ or $b,a$ forms a $d$-sequence in $A$. Hence every parameter ideal in  $A$  is generated by a $d$-sequence of length $2$.
\item[$(2)$] $\Lambda_1(A) = \{-\frac{(2\ell - n +1)n}{2} \mid 0 < n \le \ell \}$ and $\Lambda_2(A) = \{0\}$.
\end{enumerate}
\end{ex}

\begin{proof} Let $\fkm$ denote the maximal ideal in $A$ and let $x, y, z$, $w$ be the images of $X, Y, Z$,  $W$ in $A$ respectively; hence $\fkm = (x, y, z, w)$. Thanks to the exact sequence 
$$0 \to A \to A/(x,y)^\ell \oplus A/(z,w) \to A/[(x,y)^\ell + (z,w)] \to 0$$
of $A$--modules, 
we have $\operatorname{dim} A=2$, $\operatorname{depth} A = 1$, and $\H_{\fkm}^1(A) \cong A/[(x,y)^\ell + (z,w)]$. Hence $A$ is a generalized Cohen-Macaulay ring. We put  $C = A/[(x,y)^\ell + (z,w)]$.

Let us now  choose a system $a, b$ of parameters in $A$ and  put $Q = (a,b)$. Then $a, b$ are non-zerodivisors in $A$. Firstly suppose that $aC= (0)$. Hence $b, a$ is a $d$--sequence in $A$ by Proposition \ref{4.7}. If $bC = (0)$, then $QC = (0)$ and we get $$\e_Q^1(A) = - \ell_A(C) = - \frac{(\ell+1)\ell}{2} \ \ \text{and} \ \ \e_Q^2(A) = 0$$ by Proposition \ref{GhHV} (4). Suppose $bC \ne (0)$.  Let $\m_C$ be the maximal ideal of $C$ and let $$n = v_{\fkm_C}(\overline{b}) = \max \{n \in \Bbb Z \mid \cl{b} \not\in \m_C^n\}$$ denote the order of the image $\overline{b}$ of $b $  in $C$ with respect to $\m_C$. Then $0 < n < \ell$ and $(0):_C b = \fkm_C^{\ell - n}$, so that 
$$\e_Q^1(A) = -\ell_A ((0):_C b) = -\ell_A(\fkm_C^{\ell - n}) = -\frac{(2\ell -n + 1)n}{2}\ \ \text{and} \ \ \e_Q^2(A) =0$$ by Proposition \ref{GhHV} (3)

Suppose now that $aC \ne (0)$ and $bC \ne (0)$. We may assume that  $$n = v_{\fkm_C}(\overline{a}) \le  m = v_{\fkm_C}(\overline{b}).$$ Then $$b[(0):_Ca]\subseteq \fkm_C^m{\cdot}\fkm_C^{\ell - n} \subseteq \fkm_C^{\ell} = (0),$$  so that $b[(a):b^2] \subseteq b\rmU(a) \subseteq (a)$, because $\rmU (a)/(a) \cong (0):_Ca$. Thus  $a,b$ forms  a $d$-sequence in $A$. We also have  
\begin{eqnarray*}T_q &=& \left\{\left.\begin{pmatrix}
\alpha_0\\
\alpha_1\\
\vdots\\
\alpha_{q}
\end{pmatrix} \in C^{q+1} \right| a\alpha_{i} + b\alpha_{i-1}= 0\ \operatorname{for~all} \ 0 \le i \le q+1
\right\}\\
&=&\left[(0):_Ca\right]^{q+1}
\end{eqnarray*} 
for all $q \ge 0$, since $b[(0):_Ca]=(0)$.
 Therefore $$\e_Q^1(A) = -\ell_A((0):_Ca) = -\frac{(2\ell - n + 1)n}{2}\ \ \text{and} \ \ \e_Q^2(A) = 0$$ by Proposition \ref{GhHV} (3). Hence $\Lambda_2(A) = \{0 \}$.

Lastly, let $0 < n < \ell$ be integers and look at the system $a = x^\ell-z, b = y^n -w$ of parameters in $A$.  Then $aC = (0)$, $bC \ne (0)$, and $ v_{\fkm_C}(\overline{b})= n$, which shows $$\Lambda_1(A) = \{-\frac{(2\ell - n + 1)n}{2} \mid  0 < n \le \ell \}$$
as claimed.
\end{proof}

In the forthcoming paper \cite{GhGHV} we need the following.
Let us note an outline of computation.

\begin{ex}\label{2.8}
Let $\ell \ge 1$ be an integer and let $R=k[[X,Y,Z,W]]$ be the formal power series ring over a field $k$. We look at the local ring $$A = R/[(X^\ell,Y^\ell) \cap (Z,W)].$$ Then $A$ is a $2$-dimensional generalized Cohen-Macaulay local ring with $\operatorname{depth} A = 1$. Let $\fkq=(X-Z, Y-W)$ in $R$. Then $Q =\fkq A$ is a parameter ideal in $A$ with 
$$\e_Q^0(A) = \ell^2 +1, \ \ \e_Q^{1}(A) = -\ell, \ \ \text{and}\ \  \e_Q^2(A) = - \frac{\ell(\ell-1)}{2}.$$  Hence $\e_Q^2(A) < 0$ if $\ell \ge 2$, so that $Q$ cannot be generated by a $d$-sequence of length $2$ (\cite[Proposition 3.4 (2)]{GO3}).
\end{ex}

\begin{proof}
Let $C = k[X,Y, Z, W]/(X^\ell, Y^\ell, Z, W)$ and let $n \ge \ell + 1$ be an integer. We look at the graded $C$-module $$T_n = \left\{\left.\begin{pmatrix}
\alpha_0\\
\alpha_1\\
\vdots\\
\alpha_{n}
\end{pmatrix} \in C^{n+1} \right| x\alpha_{i} + y\alpha_{i-1}= 0\ \operatorname{for~all} \ 0 \le i \le n+1
\right\},$$ where $x, y$ be the images of $X, Y$ in $C$.  Let $T_{n,q}$~$(q \in \Bbb Z)$ denote the homogeneous component of degree $q$ in the graded module $T_n$.  Then $T_{n,q} =(0)$ if $q \le \ell -2$, because $(0):_C x = x^{\ell -1}C$. Let $\ell -1 \le q\le 2\ell -2$ and let $\{c_i\}_{0 \le i \le n+1}$ be a family of elements in $k$ such that  $c_i = 0$ if $n -2 \ell + q + 3 \le i \le n+1$. We put
\[ (\sharp) \ \ \alpha_i  =   \left\{
\begin{array}{ll}
\sum_{j=1}^{i+1}(-1)^{j-1}c_{i-j +1}x^{\ell -j}y^{q -\ell +  j}& \quad \mbox{if} \quad 0 \le i \le \ell -1,\\
\vspace{1mm}
\\
\sum_{j=1}^{\ell}(-1)^{j-1}c_{i - j + 1}x^{\ell  -j}y^{q -\ell +j} & \quad \mathrm{if} \quad \ell \le i \le n.
\end{array}
\right.\]
Then $
\begin{pmatrix}
\alpha_0\\
\alpha_1\\
\vdots\\
\alpha_{n}
\end{pmatrix} \in T_{n,q}$ and $T_{n,q}$ consists of all those elements which satisfy above condition  $(\sharp)$. Hence $\operatorname{dim}_kT_{n,q} = n - 2\ell + q + 3$ if $\ell -1 \le q \le 2\ell -2$. Consequently
\begin{eqnarray*}
\operatorname{dim}_kT_n &=& \sum_{q=0}^{2\ell -2}\operatorname{dim}_kT_{n,q}\\
&=&\sum_{q=\ell -1}^{2\ell -2}(n -2\ell + q +3)\\
&=& (n+1)\ell - \frac{\ell (\ell-1)}{2}.
\end{eqnarray*}
Hence $\e_Q^{1}(A) = -\ell$ and $\e_Q^2(A) = - \frac{\ell(\ell-1)}{2}$ by Proposition  \ref{GhHV} (1), which completes the computation, because  $\e_Q^{0}(A) = \e_\fkq^0(R/(X^\ell, Y^\ell)) + \e_\fkq^0(R/(Z,W)) = \ell^2 + 1$.
\end{proof}

\section{The structure of local rings of dimension $2$  possessing $\e_Q^1(A)=-1$ for some parameter ideal $Q$}

The condition that $\e_Q^1(A) = -1$ for some parameter ideal $Q$ is a rather strong restriction. In this section we shall study the structure of two--dimensional  local rings $A$ which contain such  parameter ideals. Recall that the value $\e_Q^1(A) = -1$ is the greatest among possible values of $\e_Q^1(A)$ for parameter ideals $Q$ in non-Cohen-Macaulay unmixed local rings $A$ (\cite[Theorem 2.1]{GhGHOPV1}).

Our result is the following.

\begin{thm}\label{1.1} Let $A$ be a Noetherian local ring with maximal ideal $\m$, $\operatorname{dim} A = 2$, and infinite residue class field. Assume that $\operatorname{depth} A = 1$ and that $\H_\m^1(A)$ is a finitely generated $A$-module. We  consider the following two conditions $(1)$ and $(2)$ $\mathrm{:}$
\begin{enumerate}
\item[$(1)$] 
\begin{enumerate}
\item[$(\rma)$] 
The Cohen-Macaulayfication $B$ of $A$  in the sense of \cite {AG} is not a local ring, 
\item[$(\rmb)$]
$\mu_A(\m) = 4$, and 
\item[$(\rmc)$]
$A$ contains a parameter ideal $Q$ such that  $\e_Q^1(A) = -1$.
\end{enumerate}
\item[$(2)$] There is an isomorphism $A \cong R/[(F, Y) \cap (Z,W)]$ of rings, where 
\begin{enumerate}
\item[$(\rma)$] $R$ is a regular local ring of dimension $4$ with $X,Y,Z,W$ a regular system  of parameters and 
\item[$(\rmb)$] $F = X^n + \xi$ with $\xi \in (Z, W)$ and $n \ge 1$.
\end{enumerate}
\end{enumerate}
Then, if condition $(2)$ is satisfied, condition $(1)$ is also satisfied and  $\e_\q^1(A) = -1$ for every minimal reduction $\q$ of $\m$. When $A$ is $\m$-adically complete,  conditions $(1)$ and $(2)$ are equivalent to each other.
\end{thm}

We divide the proof of Theorem \ref{1.1} into two parts.

Firstly,  let us  prove the implication $(2) \Rightarrow (1)$ in Theorem \ref{1.1}. So, let $R$ be an arbitrary regular local ring of dimension $4$ and let $X, Y, Z, W$ be a regular system of parameters in $R$.  Let $F = X^n + \xi$ with $\xi \in (Z, W)$ and   $n \ge 1$ an integer. Then  $F, Y, Z, W$ forms a system of parameters in $R$, since $(F, Y, Z, W) = (X^n, Y, Z, W)$. We put $$A = R/[(F,Y) \cap (Z,W)]$$  and let  $\m$ be the maximal ideal in $A$. We denote by $f$, $x$, $y$, $z$, $w$ the images of $F$, $X$, $Y$, $Z$, $W$ in $A$ respectively. Then, thanks to the exact sequence
$$0 \to A \to A/(f,y) \oplus A/(z,w) \to A/(x^n, y, z, w) \to 0,$$
we get $\operatorname{depth} A = 1$ and $\H_\m^1(A) \cong A/(x^n, y, z, w)$. Hence $A$ is a generalized Cohen-Macaulay ring and $$B= A/(f,y) \times A/(z,w)$$ is the Cohen-Macaulayfication of $A$ (\cite[Theorem 1.6]{AG}). We put $\c = (x^n, y, z, w)$. Therefore  $A/\c \cong \rmH_\m^1(A)$ and $\ell_A(A/\c) = n$.

For a moment, let $a_1 = f-z$ and $a_2=y-w$. We look at the parameter ideal $Q=(a_1,a_2)$ in $A$. Then, because $(a_1) : a_2 = (a_1,z)$ and $(a_2) : a_1 = (a_2, w)$, we have $$[(a_1) : a_2] + [(a_2) : a_1] = (a_1,a_2,z,w)= \c. $$ Consequently, Theorem 1.1 in the forthcoming paper \cite{GI} shows  the following, which we will refer to in \cite{GI} also.

\begin{prop}\label{GI}
The Rees algebra $\mathcal{R}(Q^2)$ of $Q^2$  is a Gorenstein ring.
\end{prop}

In order to complete the proof of the implication $(2) \Rightarrow (1)$, we are now in a position to prove the following, which gives an affirmative answer to Problem \ref{ques} with $\overline{\q} = \m$  in the present setting.

\begin{thm}\label{3}
Assume that the residue class field of $R$ is infinite.  Then $\rme_{\q}^1(A) = -1$ for every minimal reduction $\q$ of $\m$.
\end{thm}

\begin{proof}
We may assume  $n > 1$. In fact, if $n = 1$, then $A$ is a Buchsbaum ring  (\cite{SV2}), since $\H_\m^1(A) \cong A/\m$, so that $\e_\q^1(A) = -1$ by \cite[Korollar 3.2]{Sch1}. Let $\q$ be a minimal reduction of $\m$ and write $\q = (a,b)$, where we choose the system $a, b$ of generators for $\q$ so that both $a,b$ are superficial for $A$ with respect to $\q$. Let $\overline{A} = A/(z,w)$. Then, since $\q \overline{A}$ is a reduction of the maximal ideal $\m \cl{A}$ in the two-dimensional regular local ring $\overline{A}$, we get $\q \overline{A} = \m \overline{A}$, so that  $\q + (z,w) = \m$; hence $\m = \q + \c$. Without loss of generality we may assume  $a \not\in \c + \m^2 = (x^2, y, z, w)$, so that $\ell_A(C/aC) = 1$, where $C = A/\c \cong \rmH_\m^1(A)$. Since $a$ is $A$--regular and superficial for $A$ with respect to $\q$, we have by \cite[Lemma 2.4 (1)]{GNi}
$$\e_{\q}^1(A) = \e_{\q/(a)}^1(A/(a)) = - \ell_A(\rmH_\m^0(A/(a)).$$
Thus $\e_\q^1(A) = -\ell_A((0) :_Ca) =-\ell_A(C/aC) = -1$, because 
$$\rmH_\m^0(A/(a)) \cong (0):_Ca$$
which follows from the long exact sequence 
$$0 \to \rmH_\m^0(A/(a)) \to \rmH_\m^1(A) \overset{a}{\to} \rmH_\m^1(A) \to \rmH_\m^1(A/(a)) \to \ldots$$
of local cohomology modules.
\end{proof}

\begin{rem} Choose $\xi = 0$ in the above construction of $A = R/[(F,Y) \cap (Z,W)]$, whence $F =X^n$ with  an integer $n \ge 1$. Let $1 \le \ell \le n$ be an integer and put $a = x^{\ell} -z$, $b = y - w$. Let $Q = (a,b)$. Then $Q$ is a parameter ideal in $A$ and $bC = (0)$. Therefore by Proposition \ref{GhHV} (3)  we get $$\e_Q^1(A) = - \ell_A(C/aC) = - \ell_A(A/(x^{\ell}, y, z, w))= - \ell.$$ Hence $\Lambda_1(A) = \{- \ell \mid 1 \le \ell \le n \},$ because $0 > \e_\q^1(A) \ge -\ell_A(\rmH_\m^1(A)) = -n$ for every parameter ideal $\q$ in $A$ (\cite[Lemma 2.4]{GNi}, \cite[Theorem 2.1]{GhGHOPV1}).
\end{rem}

We now  prove the implication $(1) \Rightarrow (2)$ in Theorem \ref{1.1}.

Let us  reconfirm our setting. Let $A$ be a Noetherian $complete$ local ring with maximal ideal $\m$, $\operatorname{dim} A = 2$, and infinite residue class field. Assume that $\operatorname{depth} A = 1$ and $\H_\m^1(A)$ is a finitely generated $A$-module. Let $B$ denote the Cohen-Macaulayfication of $A$ in the sense of \cite{AG}. Then $B \cong \mathrm{End}_A(\mathrm{K}_A)$ as $A$-algebras (\cite[Theorem 1.6]{AG}), where $\mathrm{K}_A$  denotes the canonical module of $A$. We look at the  exact sequence
$$(E_1) \ \ \ \ 0 \to A \overset{\varphi}{\to} B \to C\to 0$$
of $A$-modules, where $\varphi (\alpha)$ is defined, for each $\alpha \in A$, to be the homothety of $\alpha$. Then, since $\operatorname{depth}_A\mathrm{K}_A = 2$,  $B$ is a Cohen-Macaulay $A$-module with $\operatorname{dim}_AB = 2$ and $C \cong \mathrm{H}_\m^1(A)$ (see the preamble of Section 2).

\vspace{2mm}

\begin{proof}[Proof of the implication $(1) \Rightarrow (2)$ in Theorem $\ref{1.1}$]
Assume that condition (1) in Theorem \ref{1.1} is satisfied, whence $B$ is not a local ring, $\mu_A(\m) = 4$,  and $A$ contains a parameter ideal $Q = (a,b)$ with $\e_Q^1(A) = -1$. We may assume that $a,b$ are both superficial for $A$ with respect to $Q$. Then, since $a$ is $A$--regular, we have
$$\rmH_\m^0(A/(a)) \cong (0):_Ca,$$
so that  $$\ell_A((0):_Ca) = - \e_{Q/(a)}^1(A/(a)) = -\e_Q^1(A) = 1$$
 (\cite[Lemma 2.4 (1)]{GNi}). Consequently $\mu_A(C) = 1$, because $\ell_A(C/aC) = \ell_A((0):_Ca)=1$. Hence $$C \cong A/\c \ \ \text{and} \ \ \c+(a) = \m,$$ where $\c = (0):_A\rmH_\m^1(A)$. Therefore  $\mu_A(B) = 2$, thanks to exact sequence $(E_1)$. Consequently, because $B$ is not a local ring and $A$ is complete, we have a canonical decomposition 
$$B = A/\a_1 \times A/\a_2$$
of $B$, where $\a_i$ is an ideal in $A$ such that $A/\a_i$ is a two-dimensional Cohen-Macaulay local ring. Hence we have $\a_1 \cap \a_2 = (0)$ and $\a_1 + \a_2 = \c,$
thanks to the exact sequence
$$0  \to A \to A/\a_1 \oplus A/\a_2 \to A/(\a_1 + \a_2) \to 0.$$

Let $V = [\c + \m^2]/\m^2 \subseteq \m/\m^2$ and put $k = A/\m$. Then $\operatorname{dim}_{k}V \ge 3$, because $\mu_A(\m) = 4$ and $\c + (a) = \m$. Since $\a_1 + \a_2 = \c$, we may assume that $\operatorname{dim}_k[\a_2 + \m^2]/\m^2 \ge 2$. Therefore the ideal $\a_2$ contains a part $z,w$  of a minimal system of generators of the maximal ideal $\m$. We then have  $\mu_{A/(z,w)}(\m/(z,w))=2$, so that  the epimorphism
$$ A/(z,w) \to A/\a_2 \to 0$$ is  an isomorphism and $A/(z,w)$ is a regular local ring of dimension $2$, because  $\operatorname{dim} A/(z,w) \ge \operatorname{dim} A/\a_2 = 2$. Thus  $\a_2 = (z,w)$. Hence $\operatorname{dim}_k[\a_1 + \m^2]/\m^2 \ge 1$, because $\operatorname{dim}_kV \ge 3$. Choose $y \in \a_1$ so that $y, z, w$ forms a part of a minimal system of generators of $\m$ and write $\m = (x, y, z, w)$. Then $A/(y,z,w) $ is a discrete valuation ring, because $A/(z,w)$ is a two-dimensional regular local ring such that  the images of $x, y$ in $A/(z,w)$ form a regular system of parameters. Consequently, since $\c = \a_1 + \a_2 \supsetneq (y, z, w)$, we have
$$\c/ (y,z,w) = (\overline{x}^n)$$ for some $n \ge 1$, where $\overline{x}$ stands for the image of $x$ in $A/(y,z,w)$. Hence $\c = (x^n, y, z, w)$ and $n = \ell_A(A/\c)$. On the other hand, because 
$$\c/(y,z,w) = [\a_1 + (z,w)]/(y,z,w) = (\overline{x}^n),$$
we find some element $\eta \in \a_1$ so that $x^n - \eta \in (y, z, w)$. Let $$x^n - \eta = \alpha y + \beta z + \gamma w$$ with $\alpha, \beta, \gamma \in A$. We then have $x^n - f \in (z,w)$ where $f = \eta + \alpha y$. Hence $\a_1 = (f,y)$, because 
$$\c = \a_1 + \a_2 \supseteq (f, y) \oplus (z, w) \supseteq \c.$$

We now choose  a $4$--dimensional complete regular local ring $R$ with maximal ideal $\n$ and  a surjective homomorphism 
$$R \overset{\phi}{\longrightarrow} A \to 0$$
of rings. Let $X, Y, Z,$ and $W$ be elements of $R$ such that $\phi(X) = x, \phi(Y) = y, \phi(Z) = z$, and $\phi(W) = w$. Then $\n = (X, Y, Z, W)$, since $\Ker \phi \subseteq \n^2$.  Notice that $$R/(Z,W) \cong A/(z,w),$$ because $A/(z,w)$ is a two-dimensional regular local ring; hence $K= \Ker~\phi \subseteq (Z,W)$.

Let $F \in R$ such that $\phi (F) = f$. We then have $X^n -F \in (Z,W)$, because $x^n -f \in (z,w)$ and $K \subseteq (Z,W)$. Therefore $(F, Y, Z, W) = (X^n, Y, Z, W)$, so that $F, Y, Z, W$ is a system of parameters in $R$. We look at the canonical exact sequence
$$(E_2) \ \ \ \ 0 \to L \to R/(F, Y) \to A/(f,y) \to 0.$$ Then, because $z,w$ forms  a regular sequence in the two-dimensional Cohen-Macaulay ring $A/\a_1 = A/(f,y)$, from exact sequence $(E_2)$ we get the exact sequence 
$$0 \longrightarrow L/(Z,W)L \longrightarrow R/(F,Y,Z,W) \overset{\varepsilon}{\longrightarrow} A/(f,y,z,w) \longrightarrow 0,$$
in which the homomorphism $\varepsilon$ has to be an isomorphism, because 
$$\ell_R(R/(F,Y,Z,W)) = \ell_R(R/(X^n,Y,Z,W)) = n$$
and $$\ell_A(A/(f,y,z,w)) = \ell_A(A/(x^n,y,z,w)) = \ell_A(A/\c) = n.$$ Thus $L = (0)$ by Nakayama's lemma, so that $R/(F,Y) \cong A/(f,y).$
Hence $K= \Ker~\phi \subseteq (F,Y)$, so that $K \subseteq (F,Y) \cap (Z,W)$. Therefore we have  $$K = (F,Y) \cap (Z,W)$$ (recall that $(F,Y) \cap (Z,W) \subseteq K$, because $(f,y) \cap (z,w) = \a_1 \cap \a_2 = (0)$). Thus $$A \cong R/[(F,Y) \cap (Z,W)],$$ with $F = X^n + \xi$ for some $\xi \in (Z,W)$. This proves the implication $(1) \Rightarrow (2)$ in Theorem \ref{1.1} under the assumption that $A$ is $\m$--adically complete.
\end{proof}

\section{Affirmative cases}
Let $A$ be a Noetherian local ring with maximal ideal $\m$ and $d = \dim A >0$. In this section  we  study Problem \ref{ques} of whether $\e_Q^1(A)$ is independent of the choice of minimal reductions $Q$ of  $I$, where $I$ is an $\m$-primary ideal in $A$.

Let us begin with the following.

\begin{prop}\label{4.1}
Let $M$ be a finitely generated $A$-module with  $\operatorname{dim}_AM = s$ and let $Q$ and $Q'$ be parameter ideals for $M$ such that  $\overline{Q} = \overline{Q'}$ in $A$. Suppose that there exists an  exact sequence
$$0 \to L \to M \to M/L \to 0$$
of $A$-modules such that $L \ne (0)$, $\operatorname{dim}_AL < s$, and $M/L$ is a Cohen-Macaulay $A$-module. Then $$\e_Q^1(M) =\e_{Q'}^1(M).$$
\end{prop}

\begin{proof}
Passing to the ring $A/[(0):_A M]$, we may assume that $(0):_AM = (0)$, whence $s=d$ and both $Q$ and $ Q'$ are parameter ideals in $A$. Let $C = M/L$. Then, since  $C$ is a Cohen-Macaulay $A$-module with $\dim_AC = d$, we have the exact sequence
$$0 \to L/Q^{n+1}L \to M/Q^{n+1}M \to C/Q^{n+1}C \to 0$$
of $A$-modules, so that
\begin{eqnarray*}
(4) \ \ \ \
\ell_A(M/Q^{n+1}M) &=&\ell_A(C/Q^{n+1}C) +  \ell_A(L/Q^{n+1}L)\\
&=&\ell_A(C/QC)\binom{n+d}{d} + \ell_A(L/Q^{n+1}L)
\end{eqnarray*}
for $n \ge 0$. Let $t = \dim_AL$ and write
$$\ell_A(L/Q^{n+1}L) = \e_Q^0(L)\binom{n + t}{t} - \e_Q^1(L)\binom{n + t -1} {t-1} + \ldots + (-1)^t\e_Q^t(L)$$
for $n \gg 0$. We then have
\[ \e_Q^1(M)  =  \left\{
\begin{array}{ll}
- \e_Q^0(L) & \quad \mbox{if} \ \ t = d-1,
\vspace{4mm}\\
0 & \quad \mathrm{if} \ \ t < d-1.
\end{array}
\right.\]
Hence $\e_Q^1(M) =\e_{Q'}^1(M)$, because $e_Q^0(L)=e_{Q'}^0(L)$ once $\overline{Q} = \overline{Q'}$.
\end{proof}

Let $M$ be a finitely generated $A$-module. We say that $M$ is a sequentially Cohen-Macaulay $A$-module, if $M$ possesses a Cohen-Macaulay  filtration, that is a filtration  $$L_0 = (0) \subsetneq L_1 \subsetneq L_2 \subsetneq \ldots \subsetneq L_{\ell} = M$$ of $A$-submodules $\{L_i\}_{0 \le i \le \ell}$ of $M$ such that $\operatorname{dim}_AL_i > \operatorname{dim}_AL_{i-1}$ and $L_i/L_{i-1}$ is a Cohen-Macaulay $A$-module for all $1 \le i \le \ell$ (\cite{CC, GHS, Sch2, 
St}).

When $M$ is a sequentially Cohen-Macaulay $A$--module, applying Proposition \ref{4.1}, we readily get the following.

\begin{cor}\label{4.2}
Suppose that $M$ is a sequentially Cohen-Macaulay $A$-module with $\operatorname{dim}_AM > 0$ and let $Q$ and $Q'$ be parameter ideals for $M$. Then $\e_Q^1(M) = \e_{Q'}^1(M)$ if $\overline{Q} = \overline{Q'}$ in $A$.
\end{cor}

Let us explore a typical example of sequentially Cohen-Macaulay rings.

\begin{ex}[\cite{G2}]
Let $R=k[[X,Y,Z]]$ be the formal power series ring over a field $k$. We look at the two-dimensional local ring $$A = R/[(X) \cap (Y,Z)].$$ Then $A$ is a sequentially Cohen-Macaulay ring but not Cohen-Macaulay. Let $x, y, z$ be the images of $X, Y, Z$ in $A$ respectively  and put $D = A/(y,z)$ and $B = A/(x)$. Then $D$ is a discrete valuation ring  and $B$ is a two-dimensional regular local ring. Let $Q =(a,b)$ be a parameter ideal in $A$. Then, since $a,b$ forms a $B$-regular sequence, thanks to the exact sequence $0 \to D \to A \to B \to 0,$
we get $$\ell_A(A/Q^{n+1}) = \e_{QB}^0(B)\binom{n+2}{2}+ \e_{QD}^0(D)\binom{n+1}{1}$$ for all $n \ge 0$, so that $$\e_Q^0(A) = \ell_B(B/QB), \ \ \e_Q^1(A) = -\e_{QD}^0(D), \ \ \text{and}\ \  \e_Q^2(A) = 0.$$ Therefore, if $Q'$ is a parameter ideal in $A$ with $\overline{Q'} = \overline{Q}$, we always have $\e_Q^i(A) = \e_{Q'}^i(A)$ for each $0 \le i \le 2$, because $QD = Q'D$.
\end{ex}

We now come to the main result of this section. 

\begin{thm}\label{ProjCM}
Let $I$ be an $m$-primary ideal in $A$ and assume that $(\sharp)$  the scheme $\operatorname{Proj} \calR (Q)$ is locally Cohen-Macaulay for every minimal reduction $Q$ of $I$. Then $\e_Q^1(A)$ is independent of the choice of minimal reductions $Q$  of $I$.
\end{thm}

To prove Theorem \ref{ProjCM} we need the following. See \cite{G3} for the equivalence of the four conditions and also for the former part of the last assertion.

\begin{prop}[\cite{G3}]\label{3.3} Let $A$ be a Noetherian local ring with maximal ideal $\m$ and $d=\dim A >0$. Let $Q$ be a parameter ideal in $A$. We put  $\calR = \calR (Q)$ and $G = \gr_Q(A)$. Let $\M = \m \calR + \calR_+$ denote  the unique graded maximal ideal in $\calR$.  Then the following four conditions are equivalent.
\begin{enumerate}
\item[$(1)$] $\operatorname{Proj} \calR$ is a locally Cohen-Macaulay scheme.
\item[$(2)$] $\operatorname{Proj} G$ is a locally Cohen-Macaulay scheme.
\item[$(3)$] $\calR$ is a graded generalized Cohen-Macaulay ring, that is the graded local cohomology modules $\rmH_\M^i(\calR)$ of $\calR$  with respect to $\M$  are finitely generated for all $i \ne d+1$.
\item[$(4)$] $G$ is a graded generalized Cohen-Macaulay ring, that is the graded local cohomology modules $\rmH_\M^i(G)$ of $G$ with respect to $\M$ are finitely generated for all $i \ne d$.
\end{enumerate}
When this is the case,  $A$ is a generalized Cohen-Macaulay ring and every system $a_1, a_2, \ldots, a_d$ of generators of $Q$ forms a superficial sequence for $A$ with respect to $Q$.
\end{prop}

\begin{proof}Let us briefly explain the reason why the latter part in the last assertion holds true. Suppose that $\operatorname{Proj} \calR$ is a locally Cohen-Macaulay scheme. Then, since $\rmH_\M^i(G)$ are finitely generated for all $i \ne d$, every homogeneous system $g_1, g_2, \ldots, g_d$ of parameters in $G$ forms a filter regular sequence with respect to $\M$  (\cite{STC}), that is the graded $G$-modules $$[(g_1, g_2, \ldots, g_{i-1}) :_{G} g_i]/(g_1, g_2, \ldots, g_{i-1})$$ are finitely graded for all $1 \le i \le d$. Applying this observation to the linear system $f_1 = \overline{a_1t}, f_2 = \overline{a_2t}, \ldots, f_d = \overline{a_dt}$ of parameters in  $G$  where $\overline{a_it}$ denotes the image of $a_it$ in $G$,  we readily get by definition that  $a_1, a_2, \ldots. a_d$ is a superficial sequence for  $A$ with respect to $Q$. 
\end{proof}

Let us note the following also, which we need to explore Example \ref{3.5}.  See \cite[Proof of Theorem (1.1)]{G3} for the proof.

\begin{prop}[\cite{G3}]\label{3.4} Suppose that $A$ is a generalized Cohen-Macaulay ring with $\dim A = d$ and $\depth A > 0$. Let $Q$ be a parameter ideal in $A$.  Assume that the residue class field $A/\m$ of $A$ is algebraically closed. Then $\operatorname{Proj} \calR (Q)$ is a locally Cohen-Macaulay scheme, if every system $a_1, a_2, \ldots, a_d$ of generators for $Q$ forms a $d$--sequence in $A$.
\end{prop}

We are in a position to prove Theorem \ref{ProjCM}.

\begin{proof}[Proof of Theorem $\ref{ProjCM}$] 
Since $\e_Q^1(A) = - \ell_A(\rmH_\m^0(A))$ if $d = 1$ (\cite[Lemma 2.4 (1)]{GNi}), we may assume that $d > 1$ and that our assertion holds true for $d-1$. 
Let $Q = (a_1, a_2, \ldots, a_d)$ and $Q' = (b_1, b_2, \ldots, b_d)$ be two reductions of $I$ such that $a_1$ and $b_1$ are superficial for $A$ with respect to $I$. We want to show $\e_Q^1(A) = \e_{Q'}^1(A)$. To do this, let us choose an element $x \in I$ such that $x$ is superficial for $A$ with respect to $I$ and both the ideals $Q_x=(a_1, a_2, \ldots, a_{d-1}, x)$ and $Q'_x = (b_1, b_2, \ldots, b_{d-1}, x)$ are reductions of $I$. Then, in order to show $\e_Q^1(A)= \e_{Q'}^1(A)$, comparing the following four reductions of $I$
\begin{eqnarray*}
Q&=&(a_1, a_2, \ldots, a_{d-1}, a_d),\\
Q_x&=&(a_1, a_2, \ldots, a_{d-1}, x),\\
Q'_x&=&(b_1, b_2, \ldots, b_{d-1}, x),\\
Q'&=&(b_1, b_2, \ldots, b_{d-1}, b_{d}),
\end{eqnarray*} 
we may assume without loss of generality that $a=a_1 = b_1$.

We put  $B = A/(a)$ and  $J = I/(a)$. Then both $QB$ and  $Q'B$ are reductions of $J$. Because  $a$ is superficial for $A$ with respect to both the ideals $Q$ and $Q'$ by Proposition \ref{3.3}, we get 
\[ \rme_Q^{1}(A) =  \left\{
\begin{array}{ll}
\rme_{QB}^{1}(B) & \quad \mbox{if} \ \ d > 2,
\vspace{4mm}\\
\rme_{QB}^{1}(B)  + \ell_A((0):_Aa) & \quad \mathrm{if} \ \ d=2
\end{array}
\right.\]
and
\[ \rme_{Q'}^{1}(A) =  \left\{
\begin{array}{ll}
\rme_{Q'B}^{1}(B) & \quad \mbox{if} \ \ d > 2,
\vspace{4mm}\\
\rme_{Q'B}^{1}(B)  + \ell_A((0):_Aa) & \quad \mathrm{if} \ \ d=2,
\end{array}
\right.\]
which  shows, in order to see $\e_Q^1(A)= \e_{Q'}^1(A)$, it suffices to check condition $(\sharp)$ is also satisfied for the ideal $J$ in $B$.

Let $\q = (\alpha_2, \alpha_3, \ldots, \alpha_d)$ be a reduction of $J$ and let $\alpha_i = \cl{c_i}$ ($c_i \in I$), where $\cl{*}$ denotes the image in $B$. We look at the parameter ideal $Q=(c_1, c_2, \ldots , c_d)$ in $A$ with $c_1 =a$. Then, because  $J^{n+1} = \q J^n$
for all $n \gg 0$, we have $I^{n+1} \subseteq Q I^n + (a),$ so that  
$$I^{n+1}=QI^n + a[I^{n+1} : a]$$
for all $n \gg 0$.
Therefore, since $I^{n+1}  : a = I^n + [(0):a]$ if $n \gg 0$, we  have $I^{n+1} = QI^n$ for all $n \gg 0$, so that $Q$ is a reduction of $I$. We now notice that by Proposition \ref{3.3}  $c_1$ is  a  superficial element  for $A$ with respect to $Q$, because the scheme $\operatorname{Proj} \mathcal{R}(Q)$ is locally Cohen-Macaulay. Therefore, the kernel of the canonical epimorphism $$\varphi : \gr_Q(A)/c_1t{\cdot}\gr_Q(A) \to \gr_{\q}(B)$$ of graded rings is 
finitely graded, so that $\gr_{\q}(B)$ is a graded generalized Cohen-Macaulay ring, because so is the graded ring $\gr_Q(A)$. Therefore $\operatorname{Proj} \calR (\q)$ is  a locally Cohen-Macaulay scheme  by Proposition \ref{3.3}. Thus condition $(\sharp)$ is satisfied for the ideal $J$ in $B$, which completes the proof of Theorem \ref{ProjCM}.
\end{proof}

Let us explore an example which satisfies  condition $(\sharp)$  in Theorem \ref{ProjCM}.

\begin{ex}\label{3.5}
We look at the local ring 
$$A = k[[X,Y,Z,W]]/[(X,Y)^2 \cap (Z,W)],$$
where $k[[X,Y,Z,W]]$ is the formal power series ring over an algebraically closed  field $k$. Then $\dim A = 2$, $\depth A = 1$, and $A$ is a generalized Cohen-Macaulay ring possessing $$\rmH_\m^1(A) \cong A/[(x,y)^2 + (z,w)],$$ where $x, y, z, w$ denotes the image of $X, Y, Z, W$ in $A$ respectively. For this ring $A$,  every $a,b \in A$ such that $Q = (a,b)$ is a reduction of $\m$ forms a $d$--sequence in $A$, so that $\operatorname{Proj} \calR (Q)$ is locally Cohen-Macaulay. We have 
$$\e_Q^0(A) = 4,\ \  \e_Q^1(A) = -2, \ \ \text{and}\ \ \e_Q^2(A) = 0$$
for every minimal reduction $Q$ of $\m$. Hence $\Lambda (\m) = \{-2\}$.  The maximal ideal $\m$ of $A$ is not  standard (\cite[Corollary 3.7]{T}), because $\m \rmH_\m^1(A) \ne (0).$ 
\end{ex}

\begin{proof}
Let $a, b \in A$ and assume that $Q = (a,b)$ is a reduction of $\m$. Since $\depth A >0$, $a$ and $b$ are non-zerodivisors in $A$. First of all, we will show that $a,b$ is a $d$--sequence in $A$. Let $\cl{A} = A/(z,w)$. Then we have $\m = (a, b) + (z,w)$, because $\cl{A}$ is a regular local ring of dimension $2$ and $Q\cl{A}$ is a reduction of the maximal ideal $\m \cl{A}$ in $\cl{A}$. Let $\m_C$ denote the maximal ideal in $C = A/[(x,y)^2 + (z,w)]$. Then $\m_C = (\cl{a}, \cl{b})$, where $\cl{*}$ denotes the image in $C$. Because $\mu_C(\m_C) = 2$, we have  $\cl{a} \ne 0$, so that $(0) :_C a = \m_C$ and $b[(0):_Ca] =(0)$ (notice that $\m_C^2 = (0)$), while $$\rmU (a)/(a) = \rmH_\m^0(A/(a)) \cong (0) :_{\rmH_\m^1(A)} a \cong (0):_C a.$$ Hence $b[(a) : b^2] \subseteq b\rmU (a) \subseteq (a).$ Thus, for a given minimal reduction $Q$ of $\m$, every system $a,b$ of generators for the ideal $Q$ forms a $d$--sequence in $A$, which implies $\Proj \calR (Q)$ is a locally Cohen-Macaulay scheme by Proposition \ref{3.5}, because the ground field $k$ is algebraically closed. We also have 
$$\e_Q^0(A) = 4,\ \  \e_Q^1(A) = -2, \ \ \text{and}\ \ \e_Q^2(A) = 0$$ by Proposition \ref{GhHV} (1), since $(0) :_C a = \m_C$ and $b[(0):_Ca] =(0)$ (notice that $T_n = [(0):_Ca]^{n+1}$ for all $n \ge 0$ in the present case). 

\end{proof}


\section{A negative case and counterexamples}

In this section we study a negative case and give counterexamples.

Let us begin with the following.

\begin{thm}\label{thm4.5}
Suppose that $A$ is a generalized Cohen-Macaulay ring with $d = \dim A \ge 2$, $\depth A \ge 1$, and infinite residue class field. Let $Q$ be a standard parameter ideal in $A$ and put $I = \cl{Q}$. If $I$ is not standard, then there exists at least one minimal reduction $Q'$ of $I$ such that  $$0 > \e_{Q'}^1(A) > \e_Q^1(A) = - \sum_{j=1}^{d-1}\binom{d-2}{j-1}\mathrm{h}^j(A),$$
where $\mathrm{h}^j(A) = \ell_A(\rmH_\m^j(A))$ for each $1 \le j \le d-1$.
\end{thm}

\begin{proof}
We have $\e_Q^1(A) = - \sum_{j=1}^{d-1}\binom{d-2}{j-1}\mathrm{h}^j(A)$ by \cite[Korollar 3.2]{Sch1}, since $Q$ is a standard parameter ideal in $A$. 
Let $\ell = \mu_A(I)$ and write $I = (x_1, x_2, \ldots, x_{\ell})$ so that every  $d$ elements $x_{i_1}, x_{i_2}, \ldots, x_{i_d}$~$(1\le i_1  < i_2 < \ldots < i_d \le \ell)$ generate a reduction of $I$. Then, since $I$ is not standard,  for some $1\le i_1  < i_2 < \ldots < i_d \le \ell$ the ideal $Q' = (x_{i_1}, x_{i_2}, \ldots, x_{i_d})$ cannot be standard (\cite[Proposition 3.2]{T}). Therefore $\cl{Q'} = I = \cl{Q}$ but 
$$\e_{Q'}^1(A) >  - \sum_{j=1}^{d-1}\binom{d-2}{j-1}\mathrm{h}^j(A)$$
by \cite[Lemma 2.4]{GNi}) and \cite[Theorem 2.1]{GO2}. See \cite[Theorem 2.1]{GhGHOPV1} for the inequality $0 > \e_{Q'}^1(A)$. 
\end{proof}

Let us construct counterexamples. The main examples are the following.

\begin{strategy}
Let $d, n \ge 2$ be integers and  look at the local ring 
$$A = R/[(X_1^n, X_2^n, \ldots, X_d^n) \cap (Y_1, Y_2, \ldots, Y_d)],$$
where $R=k[[X_1, X_2, \ldots, X_d, Y_1, Y_2, \ldots, Y_d]]$ is the formal power series ring over a field $k$. Then $\dim A = d$, $\depth A = 1$, and $A$ is a generalized Cohen-Macaulay ring with $\rmH_\m^i(A) = (0)$ for all $i \ne 1, d$. Let $x_i$ and $y_j$ denote respectively  the images of $X_i$ and $Y_j$ in $A$. Hence $\rmH_\m^1(A) \cong A/\c,$ where $\c = (x_1^n, x_2^n, \ldots, x_d^n)+(y_1, y_2, \ldots, y_d)$. We put  $Q = (x_i^n -y_i \mid 1 \le i \le d)$. Then the parameter ideal $Q$ in $A$ is standard by \cite[Corollary 3.7]{T}, since $Q\rmH_\m^1(A) = (0)$. We have $\c = Q + (x_1^n, x_2^n, \ldots, x_d^n) = Q + (y_1, y_2, \ldots, y_d)$, so that $c^2 = Q\c + (x_1^n, x_2^n, \ldots, x_d^n)(y_1, y_2, \ldots, y_d) = Q\c$; hence $Q$ is a minimal reduction of $\c$. Because $x_1x_2^{n-1} \in \m^n \setminus \c$ (recall  that $n \ge 2$), we have $\c \subsetneq \cl{\c} = \m^n + (y_1, y_2, \ldots, y_d).$ Hence $\cl{Q} = \cl{\c}$ is not standard by \cite[Corollary 3.7]{T}.  Thus Theorem \ref{thm4.5} guarantees there exists at least one minimal reduction $Q'$ of $\cl{Q}$ such that $$0 > \e_{Q'}^1(A) > \e_Q^1(A) = -n^d,$$
provided the field $k$ is infinite. 
\end{strategy}

We  precisely explore the case where $d =2$. Let $I$ be an $\m$--primary ideal in $A$ and suppose that $Q$ is a reduction of $I$. We put $$\rmS_Q(I) = I\calR (I)/I\calR (Q) \cong \bigoplus_{n \ge 1}I^{n+1}/Q^nI$$ and call it the Sally module of $I$ with respect to $Q$ (\cite{V1}). Sally modules play a very important role in the analysis of ideals $I$ with interaction to their reductions $Q$  (see, e.g., \cite{GNO1, GNO2, GO2}), as we will also  see in the following.

\begin{ex}\label{ex4.6}
Let $R=k[[X, Y, Z, W]]$ be the formal power series ring over a field $k$ and let $\a = (X^n,Y^n) \cap (Z,W)$
with  $n \ge 2$. We look at the local ring $A = R/\a$. Then $\operatorname{dim} A = 2$, $\operatorname{depth} A = 1$, and
$$\H_\m^1(A) \cong A/(x^n, y^n,z, w),$$ where $\m$ denotes the maximal ideal of $A$ and $x, y, z$, and $w$ denote the images of $X, Y, Z$, and $W$ in $A$ respectively. We put  $Q = (x^n - z, y^n -w)$ and $Q' =(xy^{n-1} - z, x^n + y^n -w)$. Then we have the following, where $\frak{c} = (x^n, y^n, z, w)$.
\begin{enumerate}
\item[$(1)$] $\overline{Q} = \overline{Q'} =\overline{\frak{c}}= \m^n + (z,w)$.
\item[$(2)$] $\e_Q^0(A) = \e_{Q'}^0(A) = 2n^2$.
\item[$(3)$] $0 > \e_{Q'}^1(A) = -n^2 + n - 1 > \e_Q^1(A) = -n^2$. Hence $\sharp \Lambda (\c) > 1$.
\item[$(4)$] $\ell_A(A/Q^{\ell +1}) = 2n^2\binom{\ell + 2}{2} + n^2\binom{\ell + 1}{1}$ and  $\ell_A(A/{Q'}^{\ell +1}) = 2n^2\binom{\ell + 2}{2} + (n^2 - n +1)\binom{\ell + 1}{1}$ for all integers $\ell \ge 0$.
\item[$(5)$] The element $x^n + y^n -w$ is not superficial for $A$ with respect to  $Q'$, so that  the scheme $\operatorname{Proj} \mathcal R (Q')$ is not locally Cohen-Macaulay.
\item[$(6)$] Let $S=\mathrm{S}_Q(I)$ (resp. $S'=\mathrm{S}_{Q'}(I)$) denote the Sally module of $I=\overline{Q} = \overline{Q'}$ with respect to $Q$ (resp. $Q'$) and let $T = \mathcal{R} (Q)$ (resp. $T' = \mathcal{R} (Q')$) be the Rees algebra of $Q$ (resp. $Q'$). Let  $\p = \m T$ and $\p' = \m T'$. Then $$\ell_{{T_\p}}(S_\p) = \ell_{T'_{\p'}}(S'_{\p'}) + (n-1).$$
\end{enumerate}
\end{ex}

\begin{proof}
Recall that  $\overline{Q} = \overline{\c}$, since $\c^2 = Q\c$. We have $\overline{\m^n + (z,w)} = \m^n + (z,w)$, since the ideal $\m^n\cl{A}$ is integrally closed in the regular local ring $\cl{A} = A/(z,w)$. Therefore $\overline{Q} = \m^n + (z,w) = \overline{\c}$, because 
$$Q \subseteq \m^n +(z,w) = (x,y)^n + (z,w) \subseteq \overline{(x^n, y^n) + (z,w)} = \overline{\c}.$$ 
Hence  $\c \ne \overline{\c}$ and $Q' \not\subseteq \c$, because $xy^{n-1} \in \cl{\c} \setminus \c$ (recall that $n \ge 2$). Let $\p_1 = (x,y)$ and $\p_2 = (z,w)$. Then
$\operatorname{Ass} A = \{\p_1, \p_2\}$
 and, thanks to the associative formula of multiplicity, we get the equality
$$\e_\q^0(A) = \sum_{\p \in \operatorname{Ass} A}\ell_{A_\p}(A_\p)\e_{\q{\cdot}(A/\p)}^0(A/\p)$$
for every $\m$-primary ideal $\q$ in $A$. Applying this observation to the ideals $Q$ and $Q'$, we readily get that $$\e_Q^0(A) = \e_{Q'}^0(A) = 2n^2.$$
 Hence $Q'$ is also a reduction of $\overline{\c}$ by a theorem of D. Rees (\cite{R}), because $Q' \subseteq \overline{\c}$ and $\e_{\overline{\c}}^0(A)=\e_{Q'}^0(A)$. Thus  $\overline{Q} = \overline{Q'}$ but $Q' \not\subseteq \c$.
We put $C = A/\c$; hence $C \cong \H_\m^1(A)$. Therefore
$$\e_{Q'}^1(A) = -n^2 + n-1$$
by Proposition \ref{GhHV} (3), because $(x^n + y^n -w)C=(0)$ and because 
\begin{eqnarray*}
\ell_A((0):_C(xy^{n-1}-z)) &=& \ell_A(C/(xy^{n-1}-z)C)\\
& = & \ell_A(A/(x^n, y^n, xy^{n-1}, z, w))\\
& = &n^2 - n + 1.
\end{eqnarray*}
By  Proposition \ref{GhHV} (1) we also have for all integers $\ell \ge 0$
\begin{eqnarray*}
\ell_A(A/{Q'}^{\ell + 1}) &=& 2n^2\binom{\ell + 2}{2} + \ell_A((0):_C(xy^{n-1}-z)) \binom{\ell + 1}{1}\\
&=& 2n^2\binom{\ell + 2}{2} + (n^2 - n+1)\binom{\ell + 1}{1}.
\end{eqnarray*}

Similarly, by Proposition \ref{GhHV} (4) we have $$\e_Q^1(A) = -\ell_A(C) = -n^2,$$
 because $QC = (0)$, so that  by Proposition \ref{GhHV} (1) $$\ell_A(A/Q^{n+1}) = 2n^2\binom{\ell + 2}{2} + n^2\binom{\ell + 1}{1}$$ for all $\ell \ge 0$.

If $x^n + y^n -w$ were superficial for $A$ with respect to  $Q'$, by \cite[Lemma 2.4 (1)]{GNi} we must have that 
\begin{eqnarray*}
\e_{Q'}^1(A) &=& \e_{Q'/(x^n + y^n -w)}^1(A/(x^n +y^n -w))\\
&=& \ell_A((0):_{C}(x^n + y^n -w))\\ 
&=& -\ell_A(C)\\
&=& -n^2,
\end{eqnarray*}
since $x^n + y^n -w$ is $A$--regular. This is impossible, since $\e_{Q'}^1(A)=-n^2 + n -1$ and $n \ge 2$. Hence  by Proposition \ref{3.3} the scheme $\operatorname{Proj} \calR (Q')$ cannot be locally Cohen-Macaulay.

To see assertion (6), notice that by \cite[Proposition 2.5]{GO1} we get the equalities
\begin{eqnarray*}
\e_I^1(A) &=& \e_I^0(A) + \e_Q^1(A) -\ell_A(A/I) + \ell_{T_\p}(S_\p)\\
&=& \e_I^0(A) + \e_{Q'}^1(A) -\ell_A(A/I) + \ell_{{T'}_\p'}({S'}_{\p'})
\end{eqnarray*}
 for the ideal $I = \overline{Q} = \cl{Q'}$, because by Proposition  \ref{4.7} conditions $\mathrm{(C_0)}$ and  $\mathrm{(C_2)}$ required in \cite{GO1} are satisfied for the ideals $Q, Q'$, and $I$.  
\end{proof}

Let us add one more fact about Example \ref{ex4.6}, which shows, similarly as Theorem \ref{3} and Example \ref{3.5}, in Example \ref{ex4.6} the value $\e_\q^1(A)$ is constant for every minimal reduction $\q$ of $\m$. From these results one might expect  an affirmative answer to Problem \ref{ques} in the case where $I=\m $. This is, however, not the case, as we will show in Section 6.

\begin{rem}\label{rank} In Example \ref{ex4.6}  assume that  the  field  $k$ is infinite. Then $\e_\q^1(A) = -n$ for every minimal reduction $\q$ of the maximal ideal $\m$ in $A$. Hence $\Lambda (\m) = \{-n \}$.
\end{rem}

\begin{proof}
We maintain the same notation as Example \ref{ex4.6}. 
Let $\overline{A} = A/(z,w)$ and let $\overline{f}$ denote, for each $f \in A$,  the image of $f$ in $\overline{A}$. Then $\overline{A}$ is a two-dimensional regular local ring with $\overline{x},\overline{y}$ a regular system of parameters. Let $\q = (a,b)$ be a minimal reduction of $\m$. Then $\m = \q + (z,w)$, since $\q \cl{A}$ is a reduction of the maximal ideal in $\cl{A}$. We may assume that $a$ is superficial for $A$ with respect to  $\q$. Hence $$\e_\q^1(A) = \e_{\q/(a)}^1(A/(a)) = -\ell_A(\rmH_\m^0(A/(a))).$$  Therefore, because $\rmH_\m^0(A/(a)) \cong (0):_Ca$ where $C = A/(x^n, y^n,z, w)$, we get
\begin{eqnarray*}
\e_\q^1(A) &=& \e_{\q/(a)}^1(A/(a))\\
&=& -\ell_A((0) :_C a)\\
&=& -\ell_A(C/aC)\\
&=& -\ell_A(A/(x^n, y^n, z, w, a)).
\end{eqnarray*}  Let us  check that $\ell_A(A/(x^n, y^n, z, w, a)) = n$. We write $\overline{a} = \alpha \overline{x} + \beta \overline{y}$ with $\alpha, \beta \in \overline{A}$. We may assume that  $\alpha$ is a unit of $\overline{A}$, because $a \not\in \m^2 + (z,w)$. Therefore $$(\overline{x}^n, \overline{y}^n, \overline{a}) = (\overline{x}^n, \overline{y}^n, \overline{x} + \beta' \overline{y})= (\overline{y}^n, \overline{x} + \beta' \overline{y})$$
with $\beta' = \alpha^{-1}\beta$. Thus  $$\ell_A(A/(x^n, y^n, z, w, a)) = \ell_A(\overline{A}/(\overline{x} + \beta' \overline{y}, \overline{y}^n)) = n.$$ Hence $\e_\q^1(A) = -n$ as claimed.
\end{proof}

Before closing this section, let us note the following, which shows a strange phenomenon that the rank of  Sally modules depends on the choice of minimal reductions.

\begin{ex}\label{sally}
Suppose that $n = 2$ in Example \ref{ex4.6} and put  $I = \m^2 + (z,w)$. Let $Q$ and $Q'$ be the same as in Example \ref{ex4.6}, whence $I = \cl{Q} = \cl{Q'}$.  We denote by $S=\mathrm{S}_Q(I)$ (resp. $S'=\mathrm{S}_{Q'}(I)$)  the Sally module of $I$ with respect to $Q$ (resp. $Q'$). Let $T=\mathcal{R} (Q) = A[Qt]$ (resp. $T' =\mathcal{R} (Q')=A[Q't]$) where $t$ is an indeterminate over $A$. We put $B = T/\m T$ and $B' = T'/\m T'$. We then have the following. 
\begin{enumerate}
\item[$(1)$] $S \cong B_+$ as graded $T$-modules.
\item[$(2)$] $S' \cong B'/(x^2 + y^2 -w)t{\cdot}B'$ as graded $T'$-modules.
\item[$(3)$] $\ell_A(A/I^{n+1}) = 8\binom{n+2}{2} -2\binom{n+1}{1} -4$ for all $n \ge 1$.
\end{enumerate}
Therefore  $\operatorname{rank}_BS = 1$ but $\operatorname{rank}_{B'}S' =0$.
\end{ex}

\begin{proof} (1) Let  $a = x^2 -z$ and $b = y^2 - w$. It is standard to check that $$I^2 = QI + (xyz, xyw), ~xyz \not\in Q,  ~I^3 = QI^2, ~\text{and}~\m I^2 \subseteq QI.$$ Hence $S \ne (0)$ and $\m S= (0)$, because  $S = TS_1$  and $S_1\cong I^2/QI$ (\cite[Lemma 2.1]{GO1}) where $S_1$ stands for the homogeneous component of $S$ with  degree $1$.  Therefore we have an epimorphism $$\varphi : B(-1) \to S$$ of graded $B$-modules defined by $$\varphi (\mathbf{e}_1) = \widetilde{xyzt} \ \  \text{and} \ \  \varphi (\mathbf{e}_2) = \widetilde{xywt},$$ where $\widetilde{xyzt}$ and $\widetilde{xywt}$ denote the images of $xyzt$ and $xywt$ in $S$ respectively and  $\{\mathbf{e}_1, \mathbf{e}_2 \}$ is the standard basis of $B(-1)^2$. Let $\overline{f}$ denote, for each $f \in T$, the image of $f$ in $B$. Then, since
$$b(xyz) = a(xyw) = -xyzw,$$
we see $
\overline{bt}\mathbf{e}_1 - \overline{at}\mathbf{e}_2 \in \Ker~\varphi
$. Therefore the homomorphism $\varphi$ gives rise to an epimorphism $$\overline{\varphi} :  B_+ \to S$$ of graded $B$--modules  (notice that $B = k[\overline{at}, \overline{bt}]$  and  $B_+ \cong B(-1)^2/B{\cdot}[\overline{bt}\mathbf{e}_1-\overline{at}\mathbf{e}_2]$, since $\overline{at}, \overline{bt}$ are algebraically independent over the residue class field $k = A/\m$ of $A$), which actually is an isomorphism, because $S \ne (0)$ and because by \cite[Lemma 2.3]{GO1} $S$ is a torsionfree $B$-module (recall that conditions $\mathrm{(C_0)}$ and $\mathrm{(C_2)}$ in \cite{GO1} are satisfied by  Proposition \ref{4.7}). Thus $S \cong B_+$ as graded $B$-modules. Condition $\mathrm{(C_1)}$ in \cite{GO1} is also satisfied, since $QC = (0)$  (\cite[Theorem 2.5]{T}).
Therefore by \cite[Theorem 1.3 (iii)]{GO1} we get
\begin{eqnarray*}
\ell_A(A/I^{n+1}) &=&\e_I^0(A)\binom{n+2}{2} -\left\{\e_I^0(A) + \e_Q^1(A) - \ell_A(A/I)+1\right\}\binom{n+1}{1} \\
&+& \left\{\e_Q^1(A) + \e_Q^2(A)\right\}\\
&=& 8\binom{n+2}{2}-2\binom{n+1}{1}-4
\end{eqnarray*}
for all $n \ge 1$.

(2) This time, we have $$I^2= Q'I + (xyz), ~I^3 = Q'I^2, ~\text{and}~ \m I^2 \subseteq Q'I.$$ Notice that $S' \ne (0)$, since $xyz \not\in Q'$.  Let $a' = z -xy$ and $b' = w - (x^2 + y^2)$.  We then  have $$b'(xyz) = a'(xyw) = xyzw$$ and $xyw = b'z - a'w \in {Q'}I$. Hence $b'(xyz) \in Q'^2I,$ so that we get  an epimorphism $${\varphi}' : (B'/b't{\cdot}B')(-1) \to S'$$ of graded $B'$--modules such that  ${\varphi}'(1) = \widetilde{xyzt}$, where $ \widetilde{xyzt}$ denotes the image of  $xyzt$ in $B'$.

We want to  show that ${\varphi}'$ is an isomorphism. Let $\overline{f}$ denote, for each $f \in T'$, the image of $f$ in $B'$. Suppose that $\Ker \ {\varphi}' \ne (0)$. Then the homogeneous component $[\Ker {\varphi}']_n$ of $\Ker \ {\varphi}'$ is non-zero for some integer $n$. Choose such an integer $n$ as small as possible. Then $n \ge 2$, since $S' \ne (0)$ and $\overline{a't}^{n-1} \in \Ker \ {\varphi}'$, since $B'=k[\overline{a't}, \overline{b't}]$. Therefore $${a'}^{n-1}(xyz) \in {Q'}^nI = a'{Q'}^{n-1}I + {b'}^nI.$$ We write  ${a'}^{n-1}(xyz) = a'i + {b'}^nj$ with $i \in {Q'}^{n-1}I$ and $j \in I$. We then have $$j \in (a') :{b'}^n = (a') :b',$$ since $a',b'$ is a $d$-sequence by Proposition \ref{4.7} (1). Let $b'j = a'h$ with $h \in A$. Then $h \in (b') : a' \subseteq I$ by Proposition \ref{4.7} (2) and ${a'}^{n-1}(xyz) = {a'}i + a'({b'}^{n-1}h),$ so that  $${a'}^{n-2}(xyz) = i + {b'}^{n-1}h \in {Q'}^{n-1}I,$$ because $a'$ is $A$-regular. Therefore  $$\overline{a't}^{n-2} \in [\Ker \ {\varphi}']_{n-1},$$ which contradicts the minimality of $n$. Hence $\varphi'$ is an isomorphism and so, $S' \cong B'/\overline{b't}{\cdot}B'$ as  graded $B'$-modules.
\end{proof}

\section{A counterexample in the case where $I = \m$}
Even if $I= \m$, the set $\Lambda (I)$ is not necessarily a singleton. To show this, we need a  technique of reducing problems to the case where $I = \m$, which we shall briefly explain in this section.

Let $B$ be a Noetherian local ring with maximal ideal $\n$ and assume that $B$ contains a field $k$ such that the composite map $$k \overset{\iota}{\to} B \overset{\varepsilon}{\to} B/\n$$
is bijective, where $\iota : k \to B$ denotes  the embedding  and $\varepsilon : B \to B/\n$ denotes the canonical epimorphism. Let us fix an  $\n$-primary ideal $J$ in $B$ and put $A = k + J$. Then $A$ is a $k$-subalgebra of $B$ and $B$ is a module-finite extension of $A$, because $$\ell_A(B/A) = \ell_A(B/J) - \ell_A(A/J) = \ell_B(B/J) -1.$$ Hence $A$ is a Noetherian local ring with maximal ideal $\m = J$ and $\operatorname{dim} A = \operatorname{dim} B$, thanks to Eakin-Nagata's theorem (\cite[Theorem 3.7]{M}).

Suppose now that $d = \operatorname{dim} B > 0$. Let $\q =(a_1, a_2, \ldots, a_d)B$ be a parameter ideal in $B$ and assume that $\q$ is a reduction of $J$. We put $Q = (a_1, a_2, \ldots, a_d)A.$
Then $Q$ is a reduction of $\m$; hence $Q$ is a parameter ideal in $A$. We have the canonical isomorphism between the Sally module $\mathrm{S}_Q(\m) = \bigoplus_{n \ge 1}\m^{n+1}/Q^{n}\m$ of $\m$ with respect to $Q$ and  the Sally module $\mathrm{S}_\q (J)=\bigoplus_{n \ge 1}J^{n+1}/\q^nJ$ of $J$ with respect to $\q$, because $$\m^{n+1}/Q^{n}\m= J^{n+1}/\q^nJ$$ for all $n \ge 1$. Consequently we get the following.

\begin{fact}\label{fact}
 $\mathrm{S}_Q(\m) \cong \mathrm{S}_\q(J)$ as graded $\mathcal{R} (Q)$-modules.
\end{fact}

We put $S = \mathrm{S}_Q(\m)$, $T = \mathcal{R} (Q)$, and $\p = \m T$. Then, thanks to \cite[Remark 2.6]{GO1}, we get 
$$\e_Q^1(A) + \ell_{T_\p}(S_\p) = \e_\m^1(A) -\e_\m^0(A) + 1.$$
Hence the sum $\e_Q^1(A) + \ell_{T_\p}(S_\p)$ is independent of the choice of reductions $\q =(a_1, a_2, \ldots, a_d)B$ of $J$, so that we have the following.

\begin{prop}\label{sum}
Let $\q=(a_1, a_2, \ldots, a_d)B$ and $\q' = (a_1', a_2', \ldots, a_d')B$ be parameter ideals of $B$ and assume that $\q$ and $\q'$ are reductions of $J$. We put  $Q = (a_1, a_2, \ldots, a_d)A$ and $Q' = (a_1', a_2', \ldots, a_d')A$. Then one has 
$$\e_Q^1(A) + \ell_{T_\p}(S_\p) = \e_{Q'}^1(A) + \ell_{{T'}_{\p'}}({S'}_{\p'})=\e_\m^1(A) -\e_\m^0(A) + 1,$$ where $S = \mathrm{S}_Q(\m)$, $S' = \mathrm{S}_{Q'}(\m)$, $T = \mathcal{R} (Q)$, $T' = \mathcal{R} (Q')$, $\p = \m T$, and $\p' = \m T'$.
\end{prop}

Let us note one example. 
The following example is based on Example \ref{ex4.6} and  shows that $\e_Q^1(A)$ depends on the choice of minimal reductions $Q$ even if  $\overline{Q} = \m$. It eventually shows that the rank, or  multiplicity, of Sally modules of the maximal ideal $\m$ depends on the choice of minimal reductions $Q$ as well.

\begin{ex}
Let $R=k[[X,Y,Z,W]]$ be the formal power series ring over a field $k$ and put $$B = R/[(X^2, Y^2) \cap (Z,W)].$$ Let  $J = (x,y)^2 + (z,w)$, where $x,y,z$ and $w$ denote the images of $X, Y, Z$, and $W$ in $B$ respectively. We look at the ring  $A = k + J$. Then $A$ is a Noetherian local ring with maximal ideal $\m = J$ and $B$ is a module-finite extension of $A$. Let $$Q = (x^2-z,y^2-w)A\ \ \text{and}\ \  Q' = (xy - z, x^2 + y^2 -w)A.$$  Then $Q$ and $Q'$ are minimal reductions of $\m$ such that
$$\e_{Q'}^1(A) = \e_{Q}^1(A) + 1 = -5, \ \ \ell_{T_\p}(S_\p) = 1, \  \ \operatorname{and}\ \ \ell_{{T'}_{\p'}}({S'}_{\p'}) = 0,$$ where $S = \mathrm{S}_Q(\m)$, $S' = \mathrm{S}_{Q'}(\m)$, $T = \mathcal{R} (Q)$, $T' = \mathcal{R} (Q')$, $\p = \m T$, and $\p' = \m T'$. Hence $\sharp \Lambda (\m) > 1$.
\end{ex}

\begin{proof}
Since $
\ell_A(A/\m^{n+1}) = \ell_A(B/J^{n+1}) - \ell_A(B/A)$ and $\ell_A(B/A) = \ell_B(B/J) - 1$,  by Example  \ref{sally} (3) we have
$$\ell_A(A/\m^{n+1}) = 8\binom{n+2}{2} - 2\binom{n+1}{1} - 6$$ for all $n \ge 1$, so that  $$\e_\m^0(A) = 8, \ \ \e_\m^1(A) = 2, \ \ \operatorname{and} \ \ \e_\m^2(A) = -6.$$ Therefore
$$
\e_Q^1(A) + \ell_{T_\p}(S_\p) = \e_{Q'}^1(A) + \ell_{{T'}_{\p'}}({S'}_{\p'}) =\e_\m^1(A) -\e_\m^0(A) +1 =-5
$$ by Proposition \ref{sum}.   On the other hand, thanks to Fact  \ref{fact} and Example \ref{sally} (1) (2), we see that $\ell_{T_\p}(S_\p)=1$ and $\ell_{{T'}_{\p'}}({S'}_{\p'}) =0$. Thus   $\e_{Q'}^1(A) = \e_{Q}^1(A) + 1 = -5$. 
\end{proof}

\vspace{1cm}


\end{document}